\documentclass[11pt]{amsart}
\usepackage{geometry}                
\usepackage[parfill]{parskip}
\usepackage{graphicx,xspace}
\usepackage{amssymb,bm}
\usepackage{epstopdf}
\usepackage{tikz}
\usepackage{hyperref}
\usepackage[capitalize]{cleveref}
\usepackage{url}
\usepackage{mathtools}	
\hypersetup{colorlinks=true, linkcolor=black, citecolor=blue, urlcolor=blue}

\newtheorem{lemma}{Lemma}
\newtheorem{proposition}[lemma]{Proposition}
\newtheorem{theorem}[lemma]{Theorem}
\newtheorem{definition}[lemma]{Definition}
\newtheorem{corollary}[lemma]{Corollary}

\theoremstyle{definition}
\newtheorem{remark}[lemma]{Remark}

\DeclareMathOperator{\gen}{gen}
\DeclareMathOperator{\inv}{inv}

\title[Interval Poset and Decomposition Trees]{The interval posets of permutations seen from the decomposition tree perspective}

 \author[M. Bouvel]{Mathilde Bouvel}
   \address[MB]{Université de Lorraine, CNRS, Inria, LORIA, F-54000 Nancy, France}
   \email{mathilde.bouvel@loria.fr}

 \author[L. Cioni]{Lapo Cioni}
   \address[LC]{Dipartimento di Matematica e Informatica "U. Dini", Università degli Studi di Firenze, Firenze, Italy} 
   \email{lapo.cioni@unifi.it}

 \author[B. Izart]{Benjamin Izart}
   \address[BI]{Université de Lorraine, CNRS, Inria, LORIA, F-54000 Nancy, France} 
   \email{benjamin.izart@loria.fr}
   
\begin{document}

\begin{abstract}
The interval poset of a permutation is the set of intervals of a permutation, ordered with respect to inclusion. It has been introduced and studied recently in~\cite{Bridget}. 
We study this poset from the perspective of the decomposition trees of permutations, describing a procedure to obtain the former from the latter. 
We then give alternative proofs of some of the results in~\cite{Bridget}, and we solve the open problems that it posed (and some other enumerative problems) using techniques from symbolic and analytic combinatorics.
Finally, we compute the M\"obius function on such posets.
\end{abstract}

\maketitle

\section{Introduction}

Recently, Bridget Tenner defined the \emph{interval posets}\footnote{We point out to the reader that interval posets are \emph{different} from interval orders defined elsewhere. An alternative (but longer) name that could describe interval posets more accurately may be ``interval containment posets'', in the spirit of~\cite{IntervalContainmentPoset}.} associated with permutations, 
and described some properties of these posets in~\cite{Bridget}.
We noted that the so-called \emph{decomposition trees} of permutations, which encode their \emph{substitution decomposition}, provide an alternative way of describing the set of intervals of any permutation. 
We therefore started a project to investigate the relations between decomposition trees and interval posets. 
The present paper reports on the results obtained in this project. 

While~\cite{Bridget} makes use of the substitution decomposition of permutations, it does not mention the associated decomposition trees. 
Our first contribution is to complement the original work~\cite{Bridget} by introducing decomposition trees into the picture, and demonstrating the strong relation between decomposition trees and interval posets.
Once this has been established, we can use it to solve the open questions posed in the final section of~\cite{Bridget}. 

The rest of this paper is organized as follows. 
In \cref{sec:background}, we review the necessary definitions regarding the interval posets of permutations and their embeddings in the plane, as well as the substitution decomposition and the resulting decomposition trees. 
In \cref{sec:procedure}, we present a procedure to compute an embedding of the interval poset of a permutation from its decomposition tree.
The procedure allows us to provide an answer to Question 7.3 of~\cite{Bridget}, describing the common properties of all permutations having the same interval poset. 
We then show in \cref{sec:oldResults} how the decomposition trees can be used to give alternative proofs of some of the results of~\cite{Bridget} regarding structural properties of interval posets. 
Next, \cref{sec:newResults} focuses on enumerative properties of interval posets.
In addition to solving the two enumerative questions left open in~\cite{Bridget} (Questions 7.1 and 7.2), we compute (exactly and asymptotically) the number of interval posets of any size.
These results are achieved using the decomposition trees and classical tools from combinatorics of trees.
Finally, in \cref{sec:Moebius}, we compute the M\"obius function of any interval of an interval poset.

\section{Background: Interval posets and decomposition trees}
\label{sec:background}

\subsection{The interval poset of a permutation}

In the context of this work, a permutation $\sigma = \sigma(1) \sigma(2) \dots \sigma(n)$ of size $n$ is 
a word on the alphabet $\{1, 2, \dots, n\}$ containing each letter exactly once. 
For example, $\sigma = 4\, 5\, 6\, 7\, 9\, 3\, 1\, 2\, 8$ is a permutation of size $9$. The size of a permutation $\sigma$ is denoted $|\sigma|$.

Intervals in permutations can be defined in several ways (focusing on the arguments -- a.k.a. positions -- in $\sigma = \sigma(1) \sigma(2) \dots \sigma(n)$ or on the images -- a.k.a. values).
Here, we make the same choice as in~\cite{Bridget}, and define an \emph{interval} of a permutation $\sigma = \sigma(1) \sigma(2) \dots \sigma(n)$ as an interval $[j,j+h]$ of values (for some $1\leq j\leq n$ and some $0 \leq h \leq n-j$) which is the image by $\sigma$ of an interval $[i,i+h]$ of positions (for some $1\leq i\leq n-h$). 
In other words, $[j,j+h]$ is an interval of $\sigma$ when there exists an $i$ satisfying $\{\sigma(i), \dots, \sigma(i+h)\} = [j,j+h]$. 
The singletons $\{1\},\{2\}, \dots, \{n\}$ and the set $[1,n]$ are always intervals of $\sigma$, and are called \emph{trivial}. 
The empty set is also an interval of $\sigma$, although most often we do not include it in the set of intervals of $\sigma$. 
Non-trivial and non-empty intervals are called \emph{proper}. 
Continuing our example $\sigma = 4\, 5\, 6\, 7\, 9\, 3\, 1\, 2\, 8$, its proper intervals are $[4,5],[4,6],[4,7],[5,6],[5,7],[6,7],[1,2]$ and $[1,3]$.

The inclusion relation naturally equips the set of intervals (proper ones and trivial ones) with a poset structure: the elements of this poset are the intervals, and the relation is the set inclusion. 
We can consider two versions of this poset: a first one in which the empty interval is an element (hence, the only minimal element in the poset), 
and a second one which excludes the empty interval (the minimal elements being then the singletons $\{1\}$ through $\{n\}$).

\begin{definition}
Let $\sigma$ be a permutation. We denote with $P_{\bullet}(\sigma)$ the \emph{interval poset} of $\sigma$ in which the empty interval is an element, and we denote with $P(\sigma)$ the interval poset of $\sigma$ in which the empty interval is excluded.
\end{definition}

While posets are essentially ``unordered'' objects, we follow~\cite{Bridget} and consider particular embeddings in the plane of the Hasse diagrams of these posets. 
Recall that the \emph{Hasse diagram} of a poset $P$ is 
the (directed) graph whose vertices are the elements of $P$, and whose edges are the covering relations\footnote{Covering relations are defined as follows. For $a$ and $b$ be two elements of $P$, we say that $a$ \emph{covers} $b$ when $b<a$ and there is no element $c$ of $P$ such that $b<c<a$.} in $P$. 
Hasse diagrams are drawn in the plane with $a$ being drawn higher than $b$ whenever $b<a$. For brevity, we write \emph{plane embedding of a poset} (or even just \emph{embedding of a poset} or \emph{embedded poset}) instead of \emph{embedding in the plane of the Hasse diagram of a poset}. Finally, an embedding is called \emph{planar} if the drawing has no crossing edges.

\begin{definition}\label{dfn:original_poset}
Let $\sigma$ be a permutation.
The \emph{canonical embedded poset} $\bar{P}(\sigma)$, is the plane embedding of the Hasse diagram of $P(\sigma)$ where the minimal elements appear in the order $\{1\}, \{2\}, \dots, \{n\}$ from left to right.

We define in the same fashion the \emph{canonical embedded poset} $\bar{P}_{\bullet}(\sigma)$: $\bar{P}_{\bullet}(\sigma)$ is just the drawing $\bar{P}(\sigma)$ with a new minimum smaller than all minimal elements of $\bar{P}(\sigma)$.
\end{definition}

The left part of \cref{fig:posets} shows the canonical embedding of our running example.
Clearly, in this figure as well as in general, every element of the poset represents the interval which consists of the set of values below it in the poset. Note that the drawing is planar (that is, it is drawn without crossing edges), and Tenner~\cite[Theorem 3.2]{Bridget} proved that this is true for every $\bar{P}(\sigma)$ and $\bar{P}_{\bullet}(\sigma)$.

\begin{figure}[ht]
\includegraphics[width=6cm]{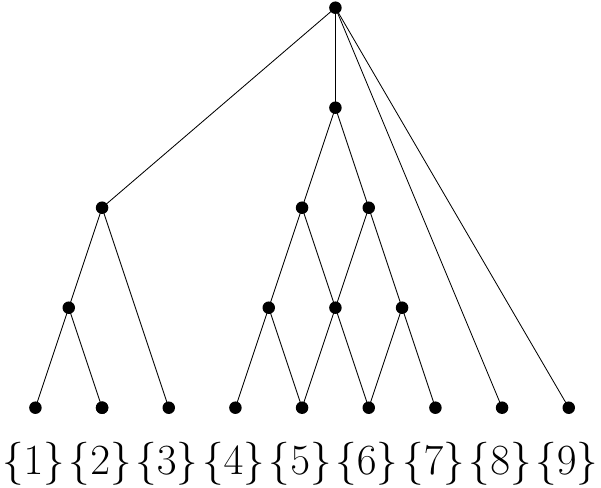} \qquad \includegraphics[width=6cm]{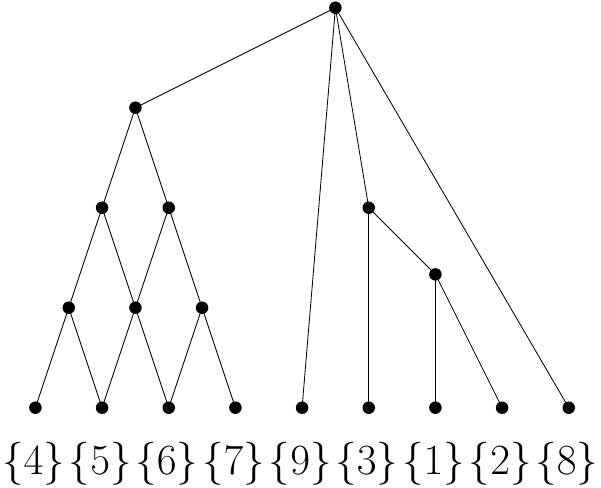}
 \caption{From left to right: The canonical embedded poset $\bar{P}(\sigma)$, and the standard embedded poset $\tilde{P}(\sigma)$, for $\sigma = 4\, 5\, 6\, 7\, 9\, 3\, 1\, 2\, 8$. 
 \label{fig:posets}}
\end{figure}

Deviating from~\cite{Bridget}, we wish to consider a different embedding of the Hasse diagram of the interval poset of a permutation, illustrated on the right part of \cref{fig:posets}.
We believe that this second embedding is also very natural, maybe even more natural than the first one once compared with the decomposition trees of permutations, as we shall see later. This embedding is also planar, but we will only prove the planarity after showing its relation with decomposition trees in \cref{sec:procedure}.

\begin{definition}
Let $\sigma$ be a permutation. 
The \emph{canonical embedded poset} $\tilde{P}(\sigma)$ is the plane embedding of the Hasse diagram of $P(\sigma)$ where the minimal elements appear in the order $\{\sigma(1)\}, \{\sigma(2)\}, \dots, \{\sigma(n)\}$ from left to right. 

We also define $\tilde{P}_{\bullet}(\sigma)$ adding a new minimum, as in \cref{dfn:original_poset}.
\end{definition}

Once again, note that the ``embedded posets'' $\bar{P}(\sigma)$ and $\tilde{P}(\sigma)$ are actually different plane drawings of the same Hasse diagram of the poset $P(\sigma)$.
They can be connected by the following remark, which is illustrated comparing the left parts of~\cref{fig:posets,fig:poset_tree}.

Let us first introduce notation. For any permutation $\sigma$, 
let us denote by $\inv(\bar{P}(\sigma))$ the poset obtained by relabeling every node of $\bar{P}(\sigma)$ corresponding to the interval $[i,j]$  by the interval $\{\sigma^{-1}(i), \dots \sigma^{-1}(j)\}$. 

\begin{proposition}\label{prop:inverse}
For any permutation $\sigma$, it holds that 
$\inv(\bar{P}(\sigma)) = \tilde{P}(\sigma^{-1})$ as plane embeddings. 
Consequently, we also have $\inv(\bar{P}_{\bullet}(\sigma))=\tilde{P}_{\bullet}(\sigma^{-1})$ (with obvious extension of the notation $\inv$).
\end{proposition}

\begin{proof}
First, observe that there is a bijective correspondence between the intervals of $\sigma$ and those of $\sigma^{-1}$. 
Namely, every interval $[j,j+h]$ of $\sigma$ such that $\{\sigma(i), \dots, \sigma(i+h)\} = [j,j+h]$
corresponds to the interval $[i,i+h]$ of $\sigma^{-1}$ -- indeed, 
$[i,i+h] = \{\sigma^{-1}(j), \dots, \sigma^{-1}(j+h)\}$. 
Next observe that the inclusion relation is preserved by this correspondence: 
for $I$ and $J$ two intervals of $\sigma$, and $I'$ and $J'$ the corresponding intervals of $\sigma^{-1}$, it holds that $I \subseteq J$ if and only if $I' \subseteq J'$. 
Therefore, the posets $P(\sigma)$ and $P(\sigma^{-1})$ are isomorphic, and the relabeling function used to witness the isomorphism is $\inv$.

We are left with proving that the embeddings of $\bar{P}(\sigma)$ and $\tilde{P}(\sigma^{-1})$ are the same. 
For this, we identify the minimal elements of $\bar{P}(\sigma)$ and $\tilde{P}(\sigma^{-1})$ by a pair $(i,j)$ where $i$ is the position of this element and $j$ its value. 
Therefore, in $\bar{P}(\sigma)$, the $i$-th minimal element in the left-to-right order has value $i$ hence corresponds to the pair $(\sigma^{-1}(i),i)$. 
On the other hand, in $\tilde{P}(\sigma^{-1})$, the $i$-th minimal element in the left-to-right order is at position $i$ in $\sigma^{-1}$, hence also corresponds to the pair $(\sigma^{-1}(i),i)$, concluding the proof. 
\end{proof}

In the present paper, we will use the standard embedded posets $\tilde{P}(\sigma)$, but \cref{prop:inverse} then of course allows to interpret them on the original canonical embedded posets $\bar{P}(\sigma)$ by considering the inverse permutation. For example, we will enumerate all the possible standard embeddings of interval posets with $n$ minimal elements, which is the same as enumerating all canonical embeddings, which is the same as enumerating all the interval posets.

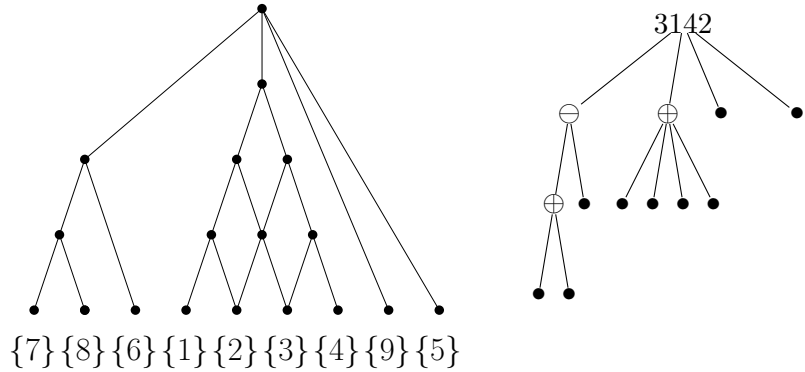
\begin{figure}[ht]
\includegraphics[width=6cm]{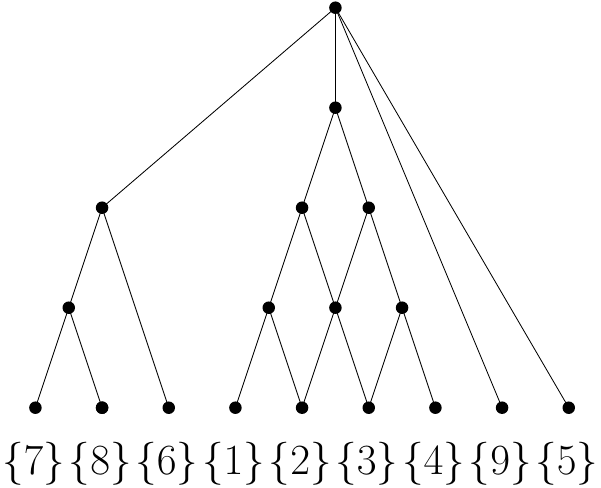} \qquad 
\begin{tikzpicture}[baseline=-4.6cm, every node/.style={inner sep=0},level distance=12mm, sibling distance=10mm] \node {$3142$}
child {node {$\ominus$} 
    child[sibling distance = 4mm] {node {$\oplus$} 
        child {node {$\bullet$} }
        child {node {$\bullet$} } 
    } 
    child[sibling distance = 4mm] {node {$\bullet$} 
    }
} 
child[sibling distance = 4mm] {node {$\oplus$} 
        child {node {$\bullet$} }
        child {node {$\bullet$} } 
        child {node {$\bullet$} }
        child {node {$\bullet$} } 
} 
child {node {$\bullet$} 
}
child {node {$\bullet$} 
};
\end{tikzpicture}
 \caption{From left to right: The standard embedded poset $\tilde{P}(\sigma^{-1})$, and the decomposition tree $T(\sigma^{-1})$, for $\sigma = 4\, 5\, 6\, 7\, 9\, 3\, 1\, 2\, 8$, \emph{i.e.}, $\sigma^{-1} = 7\, 8\, 6\, 1\, 2\, 3\, 4\, 9\, 5$. 
 The substitution decomposition of $\sigma^{-1}$ is indeed 
 $\sigma^{-1}= 3142[\ominus[\oplus[1,1],1],\oplus[1,1,1,1],1,1]$. \label{fig:poset_tree}}
\end{figure}

\subsection{Substitution decomposition and decomposition trees}

While interval posets of permutations have been defined and studied only recently, 
the inclusion relations among the intervals of permutations have been the subject of many studies, in the algorithmic and in the combinatorial literature. 
In both cases, the set of intervals is represented by means of a tree, 
which is called the \emph{strong interval tree} or \emph{(substitution) decomposition tree} depending on the context. 
For historical references, we refer to the introduction of~\cite[Section 3.2]{VatterSurvey}, and to our own work~\cite[Section 3]{StrongIntervalTrees} 
which in addition explains the equivalence between the algorithmic and the combinatorial approaches (we are not aware of many papers which make this equivalence explicit). 

Below, we review the definition of these trees, following essentially the combinatorial approach introduced in~\cite{AA05}. 
This is also the approach presented in the survey~\cite[Section 3.2]{VatterSurvey}. 
We need to recall some terminology. 

A permutation is \emph{simple} if it is of size at least $4$ and its only intervals are the trivial ones. For example, there are two simple permutations of size $4$ (namely, $2413$ and $3142$) and a simple permutation of size $7$ is $5247316$. 

Given $\pi$ a permutation of size $k$ and $k$ permutations $\alpha_1,  \dots, \alpha_k$, 
the \emph{inflation} of $\pi$ by $\alpha_1,  \dots, \alpha_k$ (a.k.a. \emph{substitution} of $\alpha_1,  \dots, \alpha_k$ in $\pi$), 
denoted $\pi[\alpha_1,  \dots, \alpha_k]$,
is the permutation obtained from $\pi$ by replacing each element $\pi(i)$ by an interval $I_i$, 
such that all elements in $I_i$ are larger than all elements in $I_j$ whenever $\pi(i) > \pi(j)$, 
and such that the elements of $I_i$ form a subsequence order-isomorphic to $\alpha_i$. 
For instance $312[12,231,4321] = 89 \, 231 \, 7654$. 
For any $k \geq 2$, we can write an inflation $12 \dots k [\alpha_1,  \dots, \alpha_k]$ unambiguously as $\oplus[\alpha_1,  \dots, \alpha_k]$ (and similarly, an inflation $k \dots 21 [\alpha_1,  \dots, \alpha_k]$ as $\ominus[\alpha_1,  \dots, \alpha_k]$), the value of $k$ being simply determined by the number of components in the inflation. For instance, $\oplus[1,3412,21,12]$ means $1234[1,3412,21,12] = 145237689$. 
Finally, we say that a permutation $\sigma$ is $\oplus$- (resp. $\ominus$-)indecomposable 
when there does not exist any $k\ge 2$ and $\alpha_i$ such that $\sigma = \oplus[\alpha_1,  \dots, \alpha_k]$ (resp. $\sigma = \ominus[\alpha_1,  \dots, \alpha_k]$).

\begin{theorem}{\cite{AA05}}\label{thm:AA05}
Every permutation $\sigma$ of size at least $2$ 
can be uniquely decomposed as 
$\pi[\alpha_1,  \dots, \alpha_k]$ with one of the following satisfied: 
\begin{itemize}
 \item $\pi$ is simple; 
 \item $\pi = 12 \dots k$ for some $k \geq 2$ and all $\alpha_i$ are $\oplus$-indecomposable; 
 \item $\pi = k \dots 21$ for some $k \geq 2$ and all $\alpha_i$ are $\ominus$-indecomposable. 
\end{itemize}
The description of $\sigma$ as $\pi[\alpha_1,  \dots, \alpha_k]$ satisfying the above is called the \emph{substitution decomposition} of $\sigma$. 
\end{theorem}

\begin{remark}
We recall our (somewhat unusual) convention that simple permutations are of size at least $4$, thus the decomposition of \cref{thm:AA05} is unique even for permutations of size 2.
Also note that the statement in the second bullet point of \cref{thm:AA05} differ slightly from~\cite{AA05}. Indeed, in the latter it is rather $\pi = \oplus[\alpha,\beta]$ with only $\alpha$ $\oplus$-indecomposable.
The two statements are easily seen to be equivalent, decomposing recursively inside $\beta$ until reaching a second component in $\oplus$ which is itself $\oplus$-indecomposable. 
A similar remark obviously applies to the third statement. 
\end{remark}

Applying the substitution decomposition recursively inside the $\alpha_i$ of \cref{thm:AA05} until we reach permutations of size $1$, 
we can represent every permutation by a tree, called its \emph{decomposition tree}, which we denote $T(\sigma)$. 
See an example on the right side of \cref{fig:poset_tree}.

\begin{definition}
A \emph{decomposition tree} of size $n$ is a plane rooted tree (that is, a rooted tree where the children of every internal vertex are ordered from left to right) with $n$ leaves and which satisfies the following: 
\begin{itemize}
\item every internal vertex $v$ is labeled by either $\oplus$, $\ominus$, or a simple permutation;
\item if an internal vertex $v$ is labeled by a simple permutation, then the size of this permutation is equal to the number of children of $v$;
\item if an internal vertex $v$ is labeled by $\oplus$ or $\ominus$, then it has at least $2$ children;
\item finally, the tree does not contain any $\oplus$-$\oplus$ edge nor any $\ominus$-$\ominus$ edge.
\end{itemize}
\end{definition}

The following theorem follows immediately from \cref{thm:AA05} 
(the absence of $\oplus$-$\oplus$ and $\ominus$-$\ominus$ edges echoing the conditions in the second and third bullet points of \cref{thm:AA05}). 

\begin{theorem}
The correspondence between permutations and decomposition trees is a size-preserving bijection. 
Under this correspondence, $\sigma(i)$ corresponds to the $i$-th leaf of $T(\sigma)$ 
in the left-to-right order. 
\end{theorem}

We do not discuss here in details the relation between the intervals of a permutation and its decomposition tree. 
We point out to the interested readers that the nodes of $T(\sigma)$ are the so-called 
\emph{strong intervals} of $\sigma$. Those are defined as the intervals of $\sigma$ which do not overlap any other interval of $\sigma$, two intervals $I$ and $J$ being overlapping when $I \cap J$ is neither empty nor equal to $I$ or $J$. 
Details can be found in~\cite[Section 3]{StrongIntervalTrees} where references are also given. 

However, the relation between interval posets and decomposition trees -- which we present in the next section -- will hopefully clarify the link between decomposition trees and intervals of permutations, even without going back to these references. 

\section{Computing the poset from the decomposition tree}
\label{sec:procedure}

We start with a few definitions (illustrated by \cref{fig:briques}), and some easy facts from~\cite{Bridget}. 

The \emph{dual claw poset of size $k$} is the poset with $k+1$ elements, of which $k$ are pairwise incomparable, and the other is larger than all of the rest. Note that it has $k$ minimal elements and, as observed in~\cite[Proposition 4.3]{Bridget}, it is (once its elements are labeled appropriately) the interval poset of all simple permutations of size $k$ (and only those). 

The \emph{argyle poset of size $k$} is the interval poset of the permutation $12 \dots k$. 
It has $k$ minimal elements. 
It was observed in~\cite[Proposition 4.4]{Bridget} that it is the interval poset of 
exactly two permutations: $12 \dots k$ and $k \dots 21$. 

\begin{figure}[ht] 
\includegraphics[width=12cm]{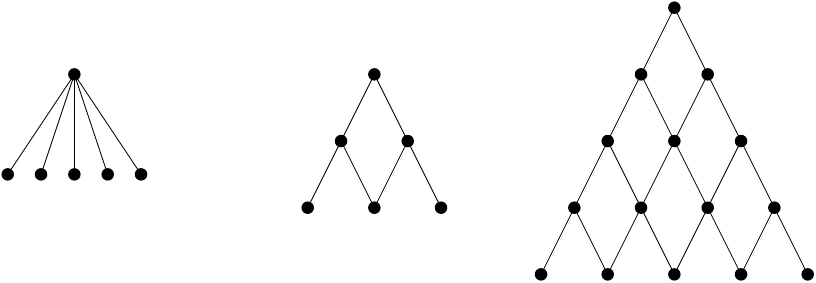}
\caption{From left to right: an embedding (canonical or standard, since they are the same) of the dual claw poset of size $5$ and of the argyle posets of size $3$ and $5$ (the labeling of their nodes has been omitted).\label{fig:briques}}
\end{figure}

Now, consider the following procedure, which takes as input a permutation $\sigma$ (or equivalently, its decomposition tree $T(\sigma)$), and returns a plane embedding of a poset (which we denote $\tilde{Q}(\sigma)$ for the moment). 
Throughout the procedure we will use an auxiliary poset $R$. To be precise, we actually use a plane embedding of the Hasse diagram of $R$, which we call $\tilde{R}$. The scope of $\tilde{R}$ is to be the top part of $\tilde{Q}(\sigma)$, so that the drawing of $\tilde{Q}(\sigma)$ can be recursively constructed by ``layers''.

\begin{enumerate}
\item If $\sigma = 1$ (\emph{i.e.} $T(\sigma) = \bullet$), we set $\tilde{Q}(\sigma)$ to be the plane embedding of the poset containing only one element.
\item Otherwise, we consider the substitution decomposition $\sigma=\pi[\alpha_1,\dots,\alpha_k]$ of $\sigma$.
\item If $\pi$ is simple, we let $R$ be the dual claw of size $k$. 
Otherwise (\emph{i.e.} if $\pi$ is $\oplus$ or $\ominus$) we let $R$ be the argyle poset of size $k$.
\item We draw the plane embedding $\tilde{R}$ of the Hasse diagram of $R$ so that its minimal elements $\pi(1),\dots,\pi(k)$ appear in this order from left to right.
\item $\tilde{Q}(\sigma)$ is obtained by taking $\tilde{R}$ and replacing, for each $i\in[1,k]$, $\pi(i)$ by the recursively-obtained $\tilde{Q}(\alpha_i)$.
\end{enumerate}

We can observe that the auxiliary poset $R$ in this procedure actually is $R = P(\pi)$, and that even the plane embeddings $\tilde{R}$ and $\tilde{P}(\pi)$ coincide.

\begin{proposition}\label{prop:proc}
For every permutation $\sigma$, $\tilde{Q}(\sigma) = \tilde{P}(\sigma)$. 
\end{proposition}

In other words, the canonical embedded poset $\tilde{P}(\sigma)$ of any permutation $\sigma$ can be obtained from $T(\sigma)$ by replacing any internal node labeled by a simple permutation by a single element, and any internal node labeled by $\oplus$ or $\ominus$ having $k$ children by several elements arranged in an argyle poset structure with $k$ minimal elements. 
This is illustrated comparing the two pictures of \cref{fig:poset_tree}.

We note that it was observed in~\cite[Section 2]{Bridget} that the substitution decomposition lays the ground work for interval posets. 
Specifically, \cref{prop:proc} can actually be seen as a rephrasing of~\cite[Theorem 4.8]{Bridget}. Nevertheless, for completeness, we provide a proof of \cref{prop:proc} below. 

Along the same line, we also note that our procedure explains a resemblance  observed in~\cite[Section 6]{Bridget}, between the ``separation trees'' and the interval posets of a special class of permutations called \emph{separable permutations}. Indeed, these separation trees are a restricted version of the decomposition trees in the special case of separable permutations, so our procedure explains precisely this similarity and allows to go from one representation to the other. We discuss this class in further details in \cref{separable permutations}.

\begin{proof}[Proof of \cref{prop:proc}]
The proof is by induction on the depth (counted in number of edges) of $T(\sigma)$. 
The statement is clear when $T(\sigma)$ is of depth $0$, \emph{i.e.},  $T(\sigma) = \bullet$. 
So, let us assume that $T(\sigma)$ is of depth at least $1$, 
and let $\sigma=\pi[\alpha_1,\dots,\alpha_k]$ be the substitution decomposition of $\sigma$. 
By the inductive hypothesis, for each $i$, $\tilde{Q}(\alpha_i) = \tilde{P}(\alpha_i)$. 

If $\pi$ is simple, then every interval of $\sigma$ is either $[1,|\sigma|]$ or included in some $\alpha_i$, implying that 
$Q(\sigma) = P(\sigma)$. 

If $\pi$ is $\oplus$ or $\ominus$, the intervals of $\sigma$ are either included in some $\alpha_i$ or consist of a union of $\alpha_i$'s for a set of consecutive indices $i$. Such intervals being represented exactly by the elements of an argyle poset, this implies that 
${Q}(\sigma) = {P}(\sigma)$ 
also in this case. 

It remains to show, in these inductive cases, that the plane embeddings $\tilde{Q}(\sigma) $ and $\tilde{P}(\sigma)$ of ${Q}(\sigma)$ and ${P}(\sigma)$  are also the same. This follows easily, since the induction hypothesis implies that the minimal elements appear in the same left-to-right order. 
\end{proof}

\begin{remark}\label{rk:identification} 
With the above procedure and \cref{prop:proc}, 
it is natural to refer to the elements of $\tilde{P}_\bullet(\sigma)$ as elements of smaller interval posets. 
More precisely, denoting $\pi[\alpha_1,  \dots, \alpha_k]$ the substitution decomposition of $\sigma$, 
we can see $\tilde{P}_\bullet(\sigma)$ as the canonical embedded poset obtained by identifying the minimal elements of $\tilde{P}(\pi)$ with the maxima of the $\tilde{P}(\alpha_i)$'s, and then adding a minimum $\emptyset$. Therefore we will refer to the elements of $\tilde{P}_\bullet(\sigma)$ different from $\emptyset$ using the corresponding elements of $\tilde{P}(\pi)$ or $\tilde{P}(\alpha_i)$.
\end{remark}

While the material presented in this section is rather simple and not new, it allows us to answer one of the open questions of~\cite{Bridget}, namely Question 7.3. 
This question concerns a description of \emph{generators}: given an interval poset $P$, a \emph{generator} of $P$ is a permutation $\sigma$ such that $P(\sigma) = P$. 
These are called \emph{interval generators} in \cite{Bridget}, but we use just ``generators'' for brevity.

Given a decomposition tree $T$, we define its \emph{skeleton} $sk(T)$ as a copy of $T$ in which each internal node has been relabeled as \emph{prime} (resp. \emph{linear}), if that node in $T$ is labeled by a simple permutation (resp. $\oplus$ or $\ominus$).

\begin{proposition}\label{prop:realizers}
Let $P$ be an interval poset, and $\sigma$ be any permutation such that $P(\sigma)=P$. 
Then a permutation $\tau$ is a generator of $P$ if and only if $sk(T(\tau^{-1}))=sk(T(\sigma^{-1}))$. 
\end{proposition}

\begin{proof}
This follows immediately from \cref{prop:inverse}, \cref{prop:proc} and the description of the procedure computing $\tilde{Q}(\sigma)$, noting also that two interval posets $P_1$ and $P_2$ are equal if and only if their canonical embeddings $\bar{P_1}$ and $\bar{P_2}$ are equal. 
\end{proof}

For example, we can compute the set of generators of the poset displayed in \cref{fig:poset_tree}, left, by considering the tree on the right, computing all the possible permutation that have the same skeleton, and taking their inverse.

The skeleton has a prime root with four children, which corresponds to two simple permutations of size four ($2413$ or $3142$). The first child has a linear node followed by a linear node, so we have two possibilities ($\oplus$-$\ominus$ or $\ominus$-$\oplus$). Finally, the second child has just a linear node, which gives again two possibilities. 
Taking the inverse of the eight resulting permutations we obtain the eight  generators:
$\{831294567, 821394567, 831297654, 821397654,$ $456793128, 456792138, 765493128, 765492138\}$.

We note that this idea was already present in~\cite[Theorem 5.1]{Bridget}, 
but only with the purpose of computing the \emph{number} of generators of an interval poset. 
Of course, the point of view of decomposition trees allows to provide an alternative proof of~\cite[Theorem 5.1]{Bridget}, although we do not provide details here, since they are very close to the proof of Theorem 4.1 in~\cite{Bridget}. 

Some of the results of \cite{Bridget} can be seen as consequences (or special cases) of \cref{prop:realizers}. 

\begin{corollary}\label{cor:complement}
The following claims hold. 
\begin{itemize}
 \item For every $k\geq 4$, the dual claw poset of size $k$ is the interval poset of all simple permutations of size $k$, and only those (see~\cite[Proposition 4.3]{Bridget}).
 \item For every $k\geq 2$, the argyle poset of size $k$ is the interval poset of exactly two permutations: $12 \dots k$ and $k \dots 21$. (see~\cite[Proposition 4.4]{Bridget}).
 \item For every permutation $\sigma = \sigma(1) \sigma(2) \dots \sigma(n)$, 
let us denote by $\sigma^R$ its \emph{reverse} $\sigma(n) \dots \sigma(2) \sigma(1)$. It holds that $P(\sigma^R) = P(\sigma)$ (see~\cite[Lemma 2.5]{Bridget}). 
\end{itemize}
\end{corollary}

\begin{proof}
For the first item, we simply observe that the simple permutations $\sigma$ of size $k$ are exactly those such that $sk(T(\sigma))$ consists of a prime node at the root with only $k$ leaves pending under it. 

Similarly, for the second item, we observe that $sk(T(\sigma))$ consists of a linear node at the root with only $k$ leaves pending under it if and only if $\sigma = 12 \dots k$ or $k \dots 21$. 

For the third item, we actually derive the finer result that $\tilde{P}(\sigma^R)$ can be obtained from $\tilde{P}(\sigma)$ by flipping the considered plane embedding along a vertical symmetry axis. 
This follows immediately by noting that the same reflection maps $sk(T(\sigma))$ to $sk(T(\sigma^R))$. \qedhere 
\end{proof}

In addition, \cite[Theorem 4.8]{Bridget} states that interval posets are the posets which can be constructed starting from the 1-element poset, and recursively replacing minimal elements with dual claw posets, argyle posets or binary tree posets (defined in \cite[Definition 4.2]{Bridget}, the latter being posets in which the Hasse diagram is a tree where every non minimal element has two children).
Our \cref{prop:proc} states that interval posets are those that can be constructed starting from the 1-element poset, and recursively replacing minimal elements with dual claw posets or argyle posets. 
Since binary tree posets can straightforwardly be obtained from the 1-element poset recursively replacing minimal elements with argyle posets with two minimal elements, 
\cref{prop:proc} therefore implies \cite[Theorem 4.8]{Bridget}. 
It also allows to identify more clearly which ``building blocks'' are needed, 
namely dual claw posets and argyle posets, without explicitly needing  binary tree posets. 

\section{Alternative proofs of known structural results}
\label{sec:oldResults}

In this section, we review several structural properties of the interval posets, which were already proved in~\cite{Bridget}. 
We believe that the approach through decomposition trees allows to provide more straightforward or elementary proofs of these statements. 

We briefly recall some classical terminology regarding properties of posets. More details can be found in~\cite{Stanley}. 
Let $P$ be a generic poset, whose partial order is denoted by~$\leq$. 

Recall that the Hasse diagram of $P$ is the (directed) graph whose vertices are the elements of $P$, and whose edges are the covering relations in $P$. 
Hasse diagrams are drawn in the plane with $a$ higher than $b$ whenever $b<a$. 
We say that $P$ is \emph{planar} when its Hasse diagram can be drawn in such a way that no two edges cross. 

For any two elements $a$ and $b$ in $P$,  
their \emph{meet} (resp. \emph{join}) denoted $a\lor b$ (resp. $a\land b$) 
is the smallest element $c$ such that $a \leq c$ and $b \leq c$. 
(resp. the largest element $c$ such that $c \leq a$ and $c \leq b$), 
if such an element exists. 
We say that $P$ is a \emph{lattice} when, for any two elements $a$ and $b$ of $P$, both $a\lor b$ and $a\land b$ exist. 
We also say that $P$ is \emph{modular} when, for any two elements $a$ and $b$ of $P$, they both cover $a\land b$ if and only if $a\lor b$ covers them both.

\begin{theorem}(\cite[Theorem 3.2]{Bridget})
For every permutation $\sigma$, the posets $P(\sigma)$ and $P_\bullet(\sigma)$ are planar. 
\end{theorem}

\begin{proof}
To prove the theorem, we consider the embeddings $\tilde{P}(\sigma)$ and $\tilde{P}_{\bullet}(\sigma)$, and prove that there is no edge crossing in these drawings.

The proof heavily relies on \cref{prop:proc} and the computation of $\tilde{P}(\sigma)$ from the decomposition tree $T(\sigma)$ of $\sigma$, which is the core of \cref{sec:procedure}. 
Clearly, for any permutation $\sigma$, the tree $T(\sigma)$ is planar, 
and the transformations performed to obtain $\tilde{P}(\sigma)$ from it maintain the planar property. 
This shows that $\tilde{P}(\sigma)$ has no crossing edges, and hence that $P(\sigma)$ is planar (in addition with a planar drawing where all minimal elements can be placed on a horizontal line at the bottom of the picture). 
Since ${P}_\bullet(\sigma)$ is obtained by adding to ${P}(\sigma)$ a new minimum smaller than all minimal elements of $P(\sigma)$, 
the above ensures that ${P}_\bullet(\sigma)$ is also planar.
\end{proof}

\begin{theorem}(\cite[Theorem 3.3]{Bridget}) \label{thm:lattice}
For every permutation $\sigma$, the poset $P_\bullet(\sigma)$ is a lattice.
\end{theorem}

\begin{proof}
First, we note the following fact: if $I$ and $J$ are two elements of some poset $P_\bullet(\sigma)$ such that $I \subseteq J$ we have $I\land J=I$ and $I\lor J=J$. We shall use this fact repeatedly, particularly when one of $I$ or $J$ is $\emptyset$. 

Now, we prove the statement by structural induction on the substitution decomposition of $\sigma$. 

If $\sigma=1$, our claim follows immediately from the above fact (since the only two elements of $P_\bullet(\sigma)$ are $\{1\}$ and $\emptyset$). 

If $\sigma$ is a simple permutation, then $P_\bullet(\sigma)$ is a dual claw poset with an added minimum~$\emptyset$. Clearly, any pair of elements have a meet and a join in this poset. 

If $\sigma$ is increasing or decreasing then $P_\bullet(\sigma)$ is an argyle poset with an added minimum. The elements of this poset correspond to the intervals $[a,b]$ for $1 \leq a \leq b \leq |\sigma|$, with the addition of $\emptyset$. 
Obviously, for all such intervals, $[a,b] \lor [a',b'] =[\min(a,a'),\max(b,b')]$ and $[a,b] \land [a',b'] =[\max(a,a'),\min(b,b')]$ with the convention that $[x,y] = \emptyset$ whenever $x>y$. 
Using also the fact observed at the beginning of the proof in the case that one of the considered element is the emptyset, it follows that for increasing or decreasing permutations $\sigma$, $P_\bullet(\sigma)$ is a lattice. 

Otherwise, we consider the substitution decomposition $\pi[\alpha_1,...,\alpha_k]$ of $\sigma$, for which it holds that $\pi \neq \sigma$. 
Note that in this case we can apply the induction hypothesis to $\pi$ as well as to each $\alpha_i$. 

Let $I$ and $J$ be two elements of $P_\bullet(\sigma)$. 
If $I$ or $J$ is $\emptyset$, then $I\lor J$ and $I\land J$ exist from the fact noted earlier. 
Therefore, let us assume that $I \neq \emptyset$ and $J \neq \emptyset$.

If $I$ and $J$ are elements of $P(\pi)$, we have $I\land J$ and $I\lor J$ in $P_\bullet(\pi)$ through the induction hypothesis. 
Then, $I\lor J$ is unchanged in $P_\bullet(\sigma)$ and $I\land J$ also stays unchanged, unless it is the minimal element ($\emptyset$) of $P_\bullet(\pi)$. 
In this case, since there is no relation between the $P(\alpha_i)$, we have $I\land J$ is the minimal element ($\emptyset$) of $P_\bullet(\sigma)$.

If $I$ and $J$ are elements of the same subposet $P(\alpha_i)$, we have $I\land J$ and $I\lor J$ in $P_\bullet(\sigma)$ through the induction 
hypothesis similarly to the previous case. 

Otherwise, $I$ and $J$ belong to different subposets $P(\alpha_i)$, and we define $I_\pi$ and $J_\pi$ the smallest elements of $P(\pi)$ that contain respectively $I$ and $J$. By transitivity, we have $I\lor J=I_\pi\lor J_\pi$ which we know exists due to the cases considered earlier.
As for $I\land J$, $I\cap J$ is empty, and thus $I\land J$ is the minimal element ($\emptyset$) of $P_\bullet(\sigma)$.

This concludes our inductive proof that for any permutation $\sigma$, $P_\bullet(\sigma)$ is a lattice.
\end{proof}

\begin{theorem}(\cite[Theorem 3.5]{Bridget})
For every permutation $\sigma$, the poset $P_\bullet(\sigma)$ is modular if and only if $\sigma$ is a simple permutation, $1$, $12$ or $21$.
\end{theorem}

While the proof of this theorem in \cite{Bridget} is based on a characterization of modularity by sublattice avoidance, 
our proof relies only on the definition of a modular lattice. 

\begin{proof}
First, let $\sigma$ be a permutation whose decomposition tree $T$ has depth at least $2$. 
Denote by $\pi[\alpha_1,...,\alpha_k]$ the substitution decomposition of $\sigma$. 
Let $I$ be an interval (of size $1$) of $\sigma$ corresponding to a leaf of $T$ at maximal depth. Let $i$ be the index such that $I$ lies in $P(\alpha_i)$. Of course, $I$ is a minimal element of $P(\sigma)$. 
Let $J$ be another minimal element of $P(\sigma)$ which lies in $P(\alpha_j)$ for some $j \neq i$. 

Then, in $P_\bullet(\sigma)$, $I  \land J = \emptyset$, which they cover. 
Defining $p_i$ and $p_j$ as the maximal elements of $P(\alpha_i)$ and $P(\alpha_j)$ respectively, 
we have $I\lor J = p_i\lor p_j$ (as in the proof of \cref{thm:lattice}). 
It is possible that $p_i\lor p_j$ covers $p_i$ and $p_j$, and it is possible that $p_j = J$. However, because $I$ corresponds to a leaf of depth at least $2$ in $T$, it holds that $p_i \neq I$. 
Therefore $I\lor J$ does not cover $I$, and $P_\bullet(\sigma)$  cannot be modular in this case.

Now, assume that $\sigma = 12 \dots k$ or $k \dots 21$ for some $k \geq 3$. Consequently, $P(\sigma)$ is an argyle poset with $k \geq 3$ minimal elements. 
Taking $I = \{1\}$ and $J= \{k\}$, 
we see that they both cover their join $\emptyset$ in $P_\bullet(\sigma)$. 
However, their meet is $[1,k]$ which does not cover them due to the argyle structure itself.

We are left with the cases $\sigma = 1, 12, 21$ or $\sigma$ is simple. 
In such cases, denoting $n = |\sigma|$, observe that all elements in $P_\bullet(\sigma)$ are either $\emptyset$, $[1,n]$, or cover the former while being covered by the latter. 
Therefore, $P_\bullet(\sigma)$ is modular, concluding the proof.
\end{proof}

Finally, in~\cite[Theorem 3.11]{Bridget} it is shown that $P_\bullet(\sigma)$ is distributive if and only if $\sigma$ is $1$ or $12$ or $21$. 
For this particular statement, decomposition trees do not allow for a proof substantially different from~\cite{Bridget}, so we leave this property outside of the present work. 

\section{Enumerative properties of interval posets}
\label{sec:newResults}

We see two possible approaches to enumeration problems on interval posets. First, given an interval poset, we can wonder how many permutations generate it. Second, like for any family of combinatorial objects, we can ask for the number of interval posets of any given size, the natural notion of size for an interval poset being its number of minimal elements. In this section, we answer (among others) these two enumeration problems, building on the understanding of interval posets gained in \cref{sec:procedure}. Indeed, \cref{prop:realizers} allows us to identify interval posets with trees in a certain family. Specifically, the following holds. 

\begin{corollary}\label{cor:IntPosetsAreTrees}
Interval posets with $n$ minimal elements are in bijection with (rooted plane) trees of the form $sk(T)$ for $T$ a decomposition tree with $n$ leaves. 
\end{corollary}

This perspective on interval posets is very useful to derive enumeration results on interval posets, using classical tools from tree enumeration.

\subsection{The number of generators of a given interval poset}

Section 5 of~\cite{Bridget} is interested in computing the number of generators of a given interval poset. 
The statement proved in~\cite{Bridget} can be rephrased in terms of decomposition trees, and we state it here for completeness. 
The proof is straightforward from the results of our \cref{sec:procedure}, and this is essentially how the statement is proved in~\cite{Bridget} (although the language is a little bit different).
Therefore, the statement is given without proof here.

\begin{theorem}(\cite[Theorem 5.1]{Bridget})
\label{thm:generatorsCount}
Let $P$ be an interval poset. Denote by $\gen(P)$ the number of generators of $P$, that is to say the number of permutations $\sigma$ such that $P = P(\sigma)$.

Let $\sigma$ be one permutation such that $P = P(\sigma)$, and let $T$ be the skeleton of the decomposition tree of $\sigma^{-1}$.
Then 
\[
 \gen(P) = \prod_{v \text{ non-leaf vertex of }T} \gen(v)^{\varepsilon_v},
\]
where the $\gen(v)$ are given by 
\[
 \gen(v) = \begin{cases}
                  2 \text{ if } v \text{ is linear},\\
                  \text{number of simple permutations of size } k \text{ if } v \text{ is prime with } k \text{ children}, \\
                 \end{cases}
\]

and the exponents $\varepsilon_v$ are given by 
\[
 \varepsilon_v = \begin{cases}
                  1 \text{ if } v \text{ is the root of } T, \\
                  1 \text{ if } v \text{ is prime}, \\
                  1 \text{ if } v \text{ is linear with a prime parent}, \\
                  0 \text{ if } v \text{ is linear with a linear parent}. \\
                 \end{cases}
\]
Equivalently, $\varepsilon_v$ is $0$ if $v$ is linear with a linear parent, $1$ otherwise. 
\end{theorem}
Note that the last case in the definition of $\varepsilon_v$ echoes the fact that there are no edges between two linear nodes with the same $\oplus$ or $\ominus$ labels in decomposition trees, leaving therefore no choice for the label of a linear node whose parent is also linear.

\subsection{Interval posets with exactly two generators}
\label{separable permutations}

In Question 7.2 of~\cite{Bridget}, B.~Tenner asked the following: 
how many interval posets have exactly two generators? 
We solve this question, and give a precise description of these interval posets.
To state our results, we need some terminology. 

A permutation is \emph{separable} if it avoids the patterns $2413$ and $3142$. The definition of pattern-avoidance is omitted here (the reader can for example refer to \cite{VatterSurvey}), because separable permutations enjoy several other characterizations, including some which are more adapted for our purpose. 
Specifically, a permutation is separable if and only if it can be obtained from permutations of size $1$ by repeated applications of the operations $\oplus$ and $\ominus$. 
This is essentially how separable permutations were first considered in~\cite{BBL}, by means of their ``separating trees''. 
One can equivalently define the separable permutations via the substitution decomposition: a permutation is separable if and only if its decomposition tree contains only $\oplus$ and $\ominus$ nodes. 

Separable permutations made an appearance in~\cite{Bridget}, where Theorem 6.2 states that $\sigma$ is separable if and only if the Hasse diagram of the interval poset $P(\sigma)$ does not contain any dual claw with more than two minimal elements.

The answer to Question 7.2 of~\cite{Bridget} actually also involves separable permutations. 

\begin{theorem}
An interval poset $P=P(\sigma)$ is such that $P$ has exactly two generators if and only if $\sigma^{-1}$ is a separable permutation of size at least $2$, or $\sigma^{-1} = 2413$ or $3142$ (equivalently, if $\sigma$ is a separable permutation of size at least $2$, or $\sigma = 2413$ or $3142$).

As a consequence, the number $a_n$ of interval posets with exactly two generators is given by the sequence $a_1 = 0$, $a_2 = s_2 =1$, $a_3= s_3 =3$, $a_4=s_4+1 =12$ and $a_n = s_n$ for all $n \geq 5$, with $s_n$ the $n$-th little Schröder number (see~\cite[sequence A001003]{OEIS}).
\end{theorem}

\begin{proof}
First observe that the interval poset $P$ corresponding to the permutation of size $1$ has just one generator. So, in the rest of the proof, we assume that $\sigma$ is of size at least two. 

Consider an interval poset $P=P(\sigma)$, and let $T$ be the skeleton of the decomposition tree of $\sigma^{-1}$. 
By \cref{thm:generatorsCount}, and using notation from this theorem, 
$\gen(P) = 2$ implies exactly one $\varepsilon_v$ is non-zero (as each $\gen(v)$ is at least $2$). So, all-but-one internal nodes must be linear with a linear parent. Since the root has no parent, it must be the root that has $\varepsilon_v = 1$. 
We now distinguish two cases depending on whether the root is prime or linear. 
The linear case yields that $\sigma^{-1}$ is a separable permutation of size at least two. 
In the prime case, to ensure that $\gen(P) = 2$, it is necessary that all children of the root are leaves and that there are four of them (indeed, there are two simple permutations of size 4, and more than two simple permutations of any size at least 5 -- see~\cite{enumerationSimple}). This yields the two other possibilities that $\sigma^{-1} = 2413$ or $3142$.
%
%
%
%
%

The parenthetical statement then follows since the class of separable permutations is closed under taking inverse, and since $2413^{-1}=3142$.

We now turn to the proof of the second part of our statement. 
The fact that $a_1 =0$ follows from the particular case of $\sigma =1$ above.
Otherwise, except when $n=4$, $a_n$ is half the number of separable permutations of size $n$. 
The number of separable permutations of size $n$ is the $n$-th large Schröder number, whose half is called little Schröder number. 
In the particular case $n=4$, we need to add $1$ to $s_n$, to account for the dual claw poset with $4$ minimal elements, whose two generators are $2413$ and $3142$. 
\end{proof}

In particular, it follows that if $\sigma$ is separable of size at least two then $P(\sigma)$ has exactly two generators, which can be described as detailed in the above proof, considering the skeleton of $T(\sigma^{-1})$.

\subsection{Counting interval posets}
\label{sec:count_int_poset}
Although the question of counting interval posets with $n$ minimal elements was neither studied nor posed in~\cite{Bridget}, we believe it is natural.
We answer this question completely in this subsection.

Let $\mathcal{P}$ be the family of rooted plane trees, where internal nodes carry a type which is either prime or linear, in which the size is defined as the number of leaves, and in which the number of children of any linear (resp. prime) node is at least $2$ (resp. at least $4$). 
Let $\mathcal{P}_n$ be the set of trees of size $n$ in $\mathcal{P}$. 
Clearly, $\mathcal{P}_n$ is the set of skeletons of decomposition trees of permutations of size $n$, and from \cref{cor:IntPosetsAreTrees} we can identify $\mathcal{P}_n$ with the set of interval posets with $n$ minimal elements.

We now use the approach of symbolic combinatorics (see~\cite[Part A]{Flajolet} for example) to obtain the enumeration of trees in $\mathcal{P}$, or equivalently of interval posets. 

By definition of $\mathcal{P}$, it follows that $\mathcal{P}$ satisfies the following combinatorial specification, where $\bullet$ denotes a leaf, $\uplus$ is the disjoint union, and $Seq_{\geq k}$ is the sequence operator restricted to sequences of at least $k$ components: 
\begin{equation*} 
\mathcal{P} = \bullet \uplus Seq_{\geq 2}(\mathcal{P}) \uplus Seq_{\geq 4}(\mathcal{P}).
\end{equation*}
Indeed, the first term corresponds to the tree consisting of a single leaf, the second to trees with a linear root, and the third to trees with a prime root. 

Denoting $p_n$ the cardinality of $\mathcal{P}_n$, 
we let $P(z) = \sum_{n\geq 0} p_n z^n$ be the ordinary generating function of $\mathcal{P}$. 
The combinatorial specification above indicates that $P(z)$ satisfies the following equation: 
\begin{equation}\label{eq:OGF_start}
 P(z) = z + \frac{P(z)^2}{1-P(z)}+ \frac{P(z)^4}{1-P(z)}.
\end{equation}

Equivalently, this can be rewritten as 
\begin{equation} \label{eq:OGF}
 P(z) = z \phi(P(z)) \text{ with } \phi(u) = \frac{1}{1-u\left(\frac{1+u^2}{1-u}\right)}.
\end{equation}

From there, the Lagrange inversion formula (see, \emph{e.g.}~\cite[Theorem A.2]{Flajolet}) can be applied to obtain an explicit formula for $p_n$. 

\begin{theorem}\label{thm:enumIntPosets}
The number $p_n$ of interval posets with $n$ minimal elements is 
\[
 p_n =
\begin{dcases}
1	&\text{ if $n=1$},\\
\frac{1}{n} \sum_{i=1}^{n-1}\sum_{k=0}^{\min\{i,\frac{n-i-1}{2}\}}\binom{n+i-1}{i}\binom{i}{k}\binom{n-2k-2}{i-1}	&\text{ if $n>1$}.
\end{dcases}
\]
\end{theorem}

The first terms of this sequence (starting from $p_1$) are $1, 1,3, 12, 52, 240, 1160, 5795, 29681$. We contributed this sequence to the OEIS~\cite{OEIS}, where it is now sequence A348479. 

\begin{proof}
Applying the Lagrange inversion theorem to \cref{eq:OGF}, we have $p_n=\frac{1}{n} [u^{n-1}]\phi(u)^n$. 
In our computation of $\phi(u)^n$, we make use of the following identity, valid for any $n \geq 1$:
\begin{equation}\label{eq:powern}
\left(\frac{1}{1-z}\right)^n = \sum_{i\ge 0}\binom{n+i-1}{i} z^i.
\end{equation}

We derive 
\begin{align}
\phi(u)^n & = \left(\frac{1}{1-u\left(\frac{1+u^2}{1-u}\right)}\right)^n =
\sum_{i\ge 0} \binom{n+i-1}{i} u^i \left(\frac{1+u^2}{1-u}\right)^i \nonumber\\
&= \sum_{i\ge 0} \binom{n+i-1}{i} u^i (1+u^2)^i \left(\frac{1}{1-u}\right)^i \nonumber\\
&= \sum_{i\ge 0} \binom{n+i-1}{i} u^i \sum_{k=0}^{i} \binom{i}{k} u^{2k} \left(\frac{1}{1-u}\right)^i  \nonumber\\
&= 1+\sum_{i> 0} \binom{n+i-1}{i} u^i \sum_{k=0}^{i} \binom{i}{k} u^{2k} \sum_{j\ge 0}\binom{i+j-1}{j} u^j \label{eq:usepowern}\\
&= 1+\sum_{i> 0} \sum_{k=0}^{i} \sum_{j\ge 0} \binom{n+i-1}{i} \binom{i}{k} \binom{i+j-1}{j} u^{i+2k+j}.\label{eq:final_phin_all}
\end{align}

The reason why we isolate the term for $i=0$ in \cref{eq:usepowern} is in order to apply \cref{eq:powern} with a positive power of $\tfrac{1}{1-u}$. 

We now want to compute $[u^{n-1}]\phi(u)^n$. 
Since $p_1 = 1$ is obvious, we can assume $n>1$. 
The exponent of $u$ in \cref{eq:final_phin_all} is $i+2k+j$, so we want $n-1 = i+2k+j$, \emph{i.e.}, $j=n-i-2k-1$. Since $j\ge0$, $i$ cannot be greater than $n-1$, while $k$ cannot be greater than $\frac{n-i-1}{2}$. Since $k$ is also at most $i$, we have $k\le \min\{i,\frac{n-i-1}{2}\}$. Therefore
\[
p_n = \frac{1}{n} [u^{n-1}]\phi(u)^n = \frac{1}{n} \sum_{i=1}^{n-1}\sum_{k=0}^{\min\{i,\frac{n-i-1}{2}\}}\binom{n+i-1}{i}\binom{i}{k}\binom{n-2k-2}{n-i-2k-1}.
\]
To conclude the proof we just note that $\binom{n-2k-2}{n-i-2k-1}=\binom{n-2k-2}{i-1}$.
\end{proof}

From \cref{eq:OGF_start}, applying the methods of analytic combinatorics (see~\cite[Part B]{Flajolet}), 
we can also derive the asymptotic behavior of $p_n$. 

\begin{theorem}\label{thm:AsympIntPosets}
Let $\Lambda$ be the function defined by $\Lambda(u) = \frac{u^2 + u^4}{1-u}$. 

The radius of convergence $\rho$ of the generating function $P(z)$ of interval posets is given by $\rho = \tau - \Lambda(\tau)$, 
where $\tau$ is the unique solution of $\Lambda'(u)=1$ such that $\tau \in (0,1)$. 

The behavior of $P(z)$ near $\rho$ is given by 
\[
P(z) = \tau - \sqrt{\frac{2 \rho}{\Lambda''(\tau)}} \sqrt{1-\frac{z}{\rho}} + \mathcal{O}\left(1-\frac{z}{\rho}\right).
\]
Numerically, we have $\tau \approx 0.2708$, $\rho \approx 0.1629$, $\sqrt{\frac{2 \rho}{\Lambda''(\tau)}}\approx 0.2206$.

As a consequence, the number $p_n$ of interval posets with $n$ minimal elements satisfies, as $n \to \infty$, 
\[
 p_n \sim \sqrt{\frac{\rho}{2 \pi \Lambda''(\tau)}} \frac{\rho^{-n}}{n^{3/2}}.
\]
 
Numerically, we have $\sqrt{\frac{\rho}{2 \pi \Lambda''(\tau)}} \approx 0.0622$, $\rho^{-1}\approx 6.1403$.
\end{theorem}

\begin{proof}
To prove this theorem we just need to prove that $\Lambda$ satisfies the hypothesis of~\cite[Theorem 1]{StrongIntervalTrees}, which is an adaptation of~\cite[Proposition IV.5 and Theorem VI.6]{Flajolet} to the setting where trees are counted by the number of \emph{leaves} (as opposed to the more classical counting by the number of \emph{nodes}). Specifically, we can immediately see from their definitions that $\Lambda$ is analytic at $0$, has non-negative Taylor coefficients, and has radius of convergence $1$, and that $P(z)$ is aperiodic. Finally, since $\lim_{u\rightarrow 1}\Lambda'(u) = +\infty>1$, the result follows immediately from~\cite[Theorem 1]{StrongIntervalTrees}.
\end{proof}

We note that the classical theorems~\cite[Proposition IV.5 and Theorem VI.6]{Flajolet} could also have been applied to obtain \cref{thm:AsympIntPosets}, starting from \cref{eq:OGF} instead of \cref{eq:OGF_start}. (And we can check that both approaches indeed yield the same results.)

\begin{remark}\label{rk:nonplane}
Recall that an interval poset consists of the set of intervals of a permutation equipped with the inclusion relation. It can be represented as a labeled drawing, in the sense that each element is labeled by an interval.

In this remark, we are interested in counting the unlabeled version of these posets, where only the underlying structure of the poset matters. 
In other words, we want to count interval posets up to isomorphism. 
We recall that two posets $P$ and $Q$ are \emph{isomorphic} (which we write $P\cong Q$) when there exists a bijection between the elements of the posets $P$ and $Q$ which preserves the order relation. For example, $P(213)\cong P(231)$, although $P(213)\neq P(231)$.
Since the relation $\cong$ is an equivalence relation, we can consider the quotient space $\mathcal{Q}$ obtained by identifying isomorphic posets from the family of interval posets, thereby defining \emph{interval posets considered up to isomorphism}, which can also be called \emph{unlabeled interval posets}. 

To count these objects we can use the same approach as for (labeled) interval posets, putting them in bijection with trees. 
The difference is that the trees to be considered are now a relaxation of the skeletons of the decomposition trees, where the plane embedding of these trees needs to be ``forgotten''. 
The correspondence between interval posets considered up to isomorphism and such trees implies that 
the family $\mathcal{Q}$ of interval posets considered up to isomorphism satisfies the following combinatorial specification, where $MSet_{\geq k}$ is the multi-set operator restricted to multi-sets of at least $k$ components:
\begin{equation*}
\mathcal{Q} = \bullet \uplus MSet_{\geq 2}(\mathcal{Q}) \uplus MSet_{\geq 4}(\mathcal{Q}).
\end{equation*}
While this specification can be translated on generating functions, the resulting equation involves Pólya operators, 
making its resolution much harder, even if just numerically. 
Nevertheless, from this equation, it can be proved that the sequence $(q_n)$ enumerating unlabeled interval posets behaves asymptotically like $\alpha n^{-3/2} \beta^{n}$, 
following the approach of \cite[Section VII.5]{Flajolet} or \cite{Harary:Twenty}. 
Iterating the equation for the generating function, we computed the first $400$ terms of the sequence $(q_n)$, which allowed us to find loose numerical estimates for $\alpha$ and $\beta$ as $\tilde{\alpha} =0.1964$ and $\tilde{\beta} = 3.7545$. 
We also contributed this sequence to the OEIS~\cite{OEIS} where its identifier is A373455. 
\end{remark}

\subsection{Counting tree interval posets}

A \emph{tree interval poset} is an interval poset whose standard (or canonical) embedded poset is a tree. 
Of course, this definition applies just to standard or canonical embedded posets $\tilde{P}(\sigma)$ and $\bar{P}(\sigma)$, since standard or canonical embedded posets $\tilde{P}_\bullet(\sigma)$ and $\bar{P}_\bullet(\sigma)$ are never trees. 

Put in our language, \cite[Theorem 6.1]{Bridget} characterizes the tree interval posets as the posets $P(\sigma)$ such that the substitution decomposition of $\sigma$ does not involve any $\oplus$ or $\ominus$ with at least three components. 
This result can be recovered from the procedure described in \cref{sec:procedure}, where we can see that an interval poset $\tilde{P}(\sigma)$ is a tree if and only if the decomposition tree of $\sigma$ has no linear node with more than two children. 
Indeed, every internal node of the tree is substituted with a dual claw or an argyle poset, and the resulting poset is a tree if and only if the substituted posets are themselves trees. 
Now, among the posets which can be substituted, the only ones that are not trees are the argyle posets with more that two minimal elements. 
Consequently, for $\tilde{P}(\sigma)$ to be a tree, the decomposition tree of $\sigma$ must be free of linear nodes with more than two children.

In~\cite[Question 7.1]{Bridget}, B.~Tenner also asks how many tree interval posets have $n$ minimal elements.
We solve this question using the same techniques as the previous subsection, giving a closed formula for the number of tree interval posets with $n$ minimal elements and the asymptotic behavior of the sequence enumerating (labeled and unlabeled) tree interval posets.

Let $\mathcal{T}$ be the family of rooted plane trees, 
where internal nodes carry a type which is either prime or linear, 
in which the size is defined as the number of leaves, 
and in which the number of children of any linear node is exactly $2$, while the number of children of any prime node is at least $4$. 
Let $\mathcal{T}_n$ be the set of trees of size $n$ in $\mathcal{T}$. 
Clearly, $\mathcal{T}_n$ is the set of skeletons of decomposition trees of permutations of size $n$ whose standard embedded poset is a tree, and we can identify $\mathcal{T}_n$ with the set of tree interval posets with $n$ minimal elements. 

We now enumerate trees in $\mathcal{T}$ using the approach of symbolic combinatorics in the same fashion as we did to enumerate trees in $\mathcal{P}$. 

Like in the previous subsection, we derive that $\mathcal{T}$ satisfies the following combinatorial specification, where $\bullet$ denotes a leaf, $\uplus$ is the disjoint union, $\times$ is the Cartesian product, and $Seq_{\geq k}$ is the sequence operator restricted to sequences of at least $k$ components: 
\begin{equation*}
\mathcal{T} = \bullet \uplus \left( \mathcal{T} \times \mathcal{T} \right) \uplus Seq_{\geq 4}(\mathcal{T}).
\end{equation*}

Denoting $t_n$ the cardinality of $\mathcal{T}_n$, 
we let $T(z) = \sum_{n\geq 0} t_n z^n$ be the ordinary generating function of $\mathcal{T}$. 
The combinatorial specification above indicates that $T(z)$ satisfies the following equation: 
\begin{equation}\label{eq:OGF_start_tree}
 T(z) = z + T(z)^2+ \frac{T(z)^4}{1-T(z)}.
\end{equation}

Equivalently, this can be rewritten as 
\begin{equation} \label{eq:OGF_tree}
 T(z) = z \psi(T(z)) \text{ with } \psi(u) = \frac{1}{1-u\left(1+\frac{u^2}{1-u}\right)}.
\end{equation}

From there, we apply again the Lagrange inversion formula to obtain an explicit formula for $t_n$. 

\begin{theorem}\label{thm:enumIntPosetsTree}
The number $t_n$ of tree interval posets with $n$ minimal elements is 
\[
t_n =
\begin{dcases}
1	&\text{ if $n=1$},\\
\frac{1}{n} \left[\sum_{i=1}^{n-3}\sum_{k=1}^{\min\{i,\frac{n-i-1}{2}\}}\binom{n+i-1}{i}\binom{i}{k}\binom{n-i-k-2}{k-1}+\binom{2n-2}{n-1}\right]	&\text{ if $n>1$}.
\end{dcases}
\]
\end{theorem}

The first terms of this sequence (starting from $t_1$) are $1, 1, 2, 6, 21, 78, 301, 1198, 4888$. This is sequence A054515 in the OEIS~\cite{OEIS}. 

\begin{proof}
The proof follows the same steps as that of~\cref{thm:enumIntPosets}.
Applying the Lagrange inversion theorem to \cref{eq:OGF_tree}, we have $t_n=\frac{1}{n} [u^{n-1}]\psi(u)^n$. 
In our computation of $\psi(u)^n$, we make use again of the identity expressed in~\cref{eq:powern}, valid for any $n \geq 1$.

We derive 
\begin{align}
\psi(u)^n & = \left(\frac{1}{1-u\left(1+\frac{u^2}{1-u}\right)}\right)^n =
\sum_{i\ge 0} \binom{n+i-1}{i} u^i \left(1+\frac{u^2}{1-u}\right)^i \nonumber\\
&= \sum_{i\ge 0} \binom{n+i-1}{i} u^i \left(1+\sum_{k=1}^{i}\binom{i}{k} \left(\frac{u^2}{1-u}\right)^k\right) \label{eq:isolate_index}\\
&= \sum_{i\ge 0} \binom{n+i-1}{i} u^i \left( 1+\sum_{k=1}^{i} \binom{i}{k} u^{2k} \sum_{j\ge 0}\binom{k+j-1}{j} u^j \right)  \nonumber\\
&= \sum_{i\ge 1} \sum_{k=1}^{i} \sum_{j\ge 0} \binom{n+i-1}{i} \binom{i}{k} \binom{k+j-1}{j} u^{i+2k+j} + \sum_{i\ge 0}\binom{n+i-1}{i} u^i.\label{eq:final_psin_all}
\end{align}

Note that we isolated the term for $k=0$ in \cref{eq:isolate_index}, in order to apply \cref{eq:powern} with a positive power of $\tfrac{1}{1-u}$. 

We now want to compute $[u^{n-1}]\psi(u)^n$. 
Since $t_1 = 1$ is obvious, we can assume $n>1$. 
The exponent of $u$ in the first term of \cref{eq:final_psin_all} is $i+2k+j$, so we want $n-1 = i+2k+j$, \emph{i.e.}, $j=n-i-2k-1$. Since $j\ge0$ and $k\ge 1$, $i$ cannot be greater than $n-3$, while $k$ cannot be greater than $\frac{n-i-1}{2}$. Since $k$ is also at most $i$, we have $k\le \min\{i,\frac{n-i-1}{2}\}$. On the other hand, the coefficient of $u^{n-1}$ is the second term of \cref{eq:final_psin_all} is $\binom{2n-2}{n-1}$. Therefore $t_n = \tfrac{1}{n}[u^{n-1}]\psi(u)^n$ gives 
\[
t_n = \frac{1}{n} \left[\sum_{i=1}^{n-3}\sum_{k=1}^{\min\{i,\frac{n-i-1}{2}\}}\binom{n+i-1}{i}\binom{i}{k}\binom{n-i-k-2}{n-i-2k-1}+\binom{2n-2}{n-1}\right].
\]
To conclude the proof we just note that $\binom{n-i-k-2}{n-i-2k-1}=\binom{n-i-k-2}{k-1}$.
\end{proof}

As in the previous subsection, we can obtain the asymptotic behavior of $t_n$ using analytic combinatorics, 
either from \cref{eq:OGF_start_tree} (which is the version we present) or from \cref{eq:OGF_tree}. 
The following can be proved exactly like \cref{thm:AsympIntPosets}. 

\begin{theorem}\label{thm:AsympIntPosetsTree}
Let $\Lambda$ be the function defined by $\Lambda(u) = u^2 + \frac{u^4}{1-u}$. 

The radius of convergence $\rho$ of the generating function $T(z)$ of tree interval posets is given by $\rho = \tau - \Lambda(\tau)$, 
where $\tau$ is the unique solution of $\Lambda'(u)=1$ such that $\tau \in (0,1)$. 

The behavior of $T(z)$ near $\rho$ is given by 
\[
T(z) = \tau - \sqrt{\frac{2 \rho}{\Lambda''(\tau)}} \sqrt{1-\frac{z}{\rho}} + \mathcal{O}\left(1-\frac{z}{\rho}\right).
\]
Numerically, we have $\tau \approx 0.3501$, $\rho \approx 0.2044$, $\sqrt{\frac{2 \rho}{\Lambda''(\tau)}}\approx 0.2808$.

As a consequence, the number $t_n$ of tree interval posets with $n$ minimal elements satisfies, as $n \to \infty$, 
\[
 t_n \sim \sqrt{\frac{\rho}{2 \pi \Lambda''(\tau)}} \frac{\rho^{-n}}{n^{3/2}}.
\]
 
Numerically, we have $\sqrt{\frac{\rho}{2 \pi \Lambda''(\tau)}} \approx 0.0792$, $\rho^{-1}\approx 4.8920$.
\end{theorem}

\begin{remark}
Like in \cref{rk:nonplane}, we can find an equation for the generating function of unlabeled tree interval posets, that is to say of tree interval posets considered up to isomorphism. 
And similarly, we can deduce from this equation that the asymptotic behavior of the sequence enumerating these objects is of the form $\alpha n^{-3/2} \beta^{n}$. 
Here, the loose numerical estimates for $\alpha$ and $\beta$ which we obtain in the same fashion as in \cref{rk:nonplane}  are $\tilde{\alpha} =0.2597$ and $\tilde{\beta} = 2.9784$. 
This sequence has been contributed to the OEIS~\cite{OEIS} as A373456. 
\end{remark}

\section{The M\"obius function on interval posets}
\label{sec:Moebius}

In this section we will calculate the M\"obius function on interval posets $P_\bullet(\sigma)$.
The very constrained structure of interval posets makes it possible to compute the M\"obius function on any interval of such a poset, a situation which is rare enough to be noted.
We first recall some basic concepts. We refer the reader to~\cite{Stanley} or~\cite{Godsil} for details. 

\begin{definition}
Let $(P,\le)$ be a partially ordered set. If $P$ has a maximum element $M$, then the elements covered by $M$ are called \emph{coatoms}.
\end{definition}

\begin{definition}
Let $(P,\le)$ be a partially ordered set and let $a, b \in P$. The \emph{interval} $[a,b]$ of $P$ is the set $[a,b]=\{x\in P\mid a\le x\le b\}$.
\end{definition}

In this paper we use the term interval to denote both the intervals of a permutation and the intervals of a poset. To avoid ambiguity, we will specify every time that we refer to the interval of a poset $P$ by writing \emph{interval of $P$}.

\begin{definition}
Let $(P,\le)$ be a partially ordered set with finitely many elements, 
and let $a, b \in P$. The \emph{M\"obius function} between $a$ and $b$ is recursively defined as
\[
\mu_P(a,b)=
\begin{dcases}
1\quad &\text{if  $a=b$},\\
-\sum_{x:a < x \le b} \mu_P(x,b)\quad &\text{if $a<b$},\\
0\quad &\text{otherwise.}
\end{dcases}
\]
\end{definition}

Whenever $P$ is clear from the context, we write just $\mu$ instead of $\mu_P$. 

Here we used the ``top-down'' definition of the M\"obius function, but we point out that the classical definition is the (obviously equivalent) ``bottom-up'' definition, given by $\mu(a,b)=-\sum_{x:a \le x < b} \mu(a,x)$, for $a<b$.
For our purpose, the top-down definition is more convenient, because in $P_\bullet(\sigma)$ it is simpler to start from the top and recursively compute the M\"obius function with the elements below.

The following lemma is an immediate consequence of~\cite[Lemma 10.4]{Godsil}, and describes a simple case where the M\"obius function is $0$. This special case arises often in~$P_\bullet(\sigma)$. We also present a brief, self-contained proof of the lemma.

\begin{lemma}[\cite{Godsil}]\label{lem:confrontabile}
Let $(P,\le)$ be a partially ordered set with finitely many elements. 
Let $[a,b]$ be an interval of $P$. If there exists an element $x\in [a,b]$, $x\neq a, b$, which is comparable with every element in the interval $[a,b]$ of $P$, then $\mu(a,b)=0$.
\end{lemma}
\begin{proof}
Let $x\in [a,b]$ be an element satisfying the properties of the statement. By definition of the M\"obius function, we have $\sum_{y:x \le y \le b}\mu(y,b) = 0$. 
Consider an element $c$ covered by $x$. Since $x$ is comparable with every element of $[a,b]$, there is no other element of $[a,b]$ which covers $c$. 
This implies that $\mu(c,b)= - \sum_{y:c < y \le b}\mu(y,b) = - \sum_{y:x \le y \le b}\mu(y,b) = 0$. Reasoning by induction, we can see that for every element $z\neq x$ in the interval $[a,x]$ of $P$ it holds that $\mu(z,b)=0$. Indeed, 
\[
\mu(z,b)= - \left(\sum_{y:z < y < x}\mu(y,b) + \sum_{y:x \le y \le b}\mu(y,b) \right) = - \left( \sum_{y:z < y < x}0 + 0 \right)= 0.
\]
In particular, $\mu(a,b)=0$.
\end{proof}

\begin{theorem}\label{thm:Moebius}
Let $\sigma$ be a permutation of size $n$ whose substitution decomposition is $\pi[\alpha_1,...,\alpha_k]$. 
For any $I\in P_\bullet (\sigma)$, it holds that 
\[
\mu(I,[1,n])=
\begin{cases}
1\quad &\text{if $I=[1,n]$},\\
-1\quad &\text{if $I$ is covered by $[1,n]$ (\emph{i.e.}, $I$ is a coatom)}, \\
k-1\quad &\text{if $I=\emptyset$ and $\pi$ is either simple or $12$ or $21$},\\
1\quad &\text{if $\pi$ is $12\dots k$ or $k \dots 21$ for some $k \geq 3$} \\ 
& \quad \text{and $I$ is covered by the two coatoms of $P_\bullet (\sigma)$},\\
0\quad &\text{otherwise.}
\end{cases}
\]
\end{theorem}

\begin{proof}
If $I=[1,n]$ then $\mu(I,[1,n])=1$, while if $I$ is a coatom then $\mu(I,[1,n])=-1$, by definition. Suppose now that $I$ is neither $[1,n]$ nor a coatom.

We now distinguish cases according to the substitution decomposition $\pi[\alpha_1,\dots,\alpha_k]$ of~$\sigma$.

If $\pi$ is simple, then $P_\bullet (\sigma)$ is obtained by identifying the $k$ minimal elements of the dual claw poset $P(\pi)$ with the maxima of $P(\alpha_i)$, for $1\le i\le k$, and adding the minimum~$\emptyset$.
This is also true when $\pi=12$ or $21$, because the argyle poset with two minimal elements is equal to the dual claw poset with two minimal elements. 
Consider an element $I$ of $P_\bullet (\sigma)$ such that $I$ is neither $\emptyset$, nor a coatom, nor $[1,n]$. 
Let $i$ be the index such that $P(\alpha_i)$ contains $I$. 
Note that $I$ is not the maximum of $P(\alpha_i)$, since it is not a coatom. 
Then, $\mu(I,[1,n])=0$ by \cref{lem:confrontabile}, where the element $x$ of the lemma is the maximum of $P(\alpha_i)$. 
Finally, we consider the case $I=\emptyset$. 
From the above results and the definition of $\mu$, we have
\begin{align*}
\mu(\emptyset,[1,n]) =& -\sum_{J\in P(\sigma)}\mu(J,[1,n]) = -\mu([1,n],[1,n]) -\sum_{J\text{ \footnotesize coatom}}\mu(J,[1,n]) \\
=& -1 -\sum_{J\text{ \footnotesize coatom}}(-1) = -1 +k.
\end{align*}

We are left with the case where $\pi$ is $12\dots k$ or $k \dots 21$ for some $k \geq 3$. 
In this case, $P_\bullet (\sigma)$ is obtained by identifying the $k$ minimal elements of the argyle poset $P(\pi)$ with the maxima of $P(\alpha_i)$, for $1\le i\le k$, and adding the minimum $\emptyset$.
We can easily compute the M\"obius function for every element $I\in P(\pi)$. 
We have seen that $\mu(I,[1,n]) = 1$ (resp. $-1$) if $I$ is $[1,n]$ (resp. a coatom). 
It follows that $\mu(I,[1,n]) = 1$ if $I$ is the element covered by both coatoms, and that $\mu(I,[1,n]) = 0$ for all the others $I\in P(\pi)$ -- see \cref{fig:moebius}. 

\begin{figure}[ht]
\includegraphics[width=4cm]{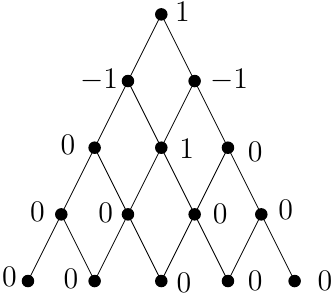} 
 \caption{The M\"obius function between the maximum and any element in an argyle poset. \label{fig:moebius}}
\end{figure}

We now consider an element $I \in P(\alpha_i)$ for some $i$. 
If $I$ is the maximum of $P(\alpha_i)$, then it is also a minimal element of $P(\pi)$, and hence $\mu(I,[1,n]) = 0$ as we have seen. 
Otherwise, we apply again \cref{lem:confrontabile} (with $x$ the maximum of $P(\alpha_i)$), and we obtain $\mu(I,[1,n])=0$.  
Finally, if $I=\emptyset$, denoting $K$ the element covered by the two coatoms of $P_\bullet (\sigma)$, then we have
\begin{align*}
\mu(\emptyset,[1,n]) =& -\sum_{J\in P(\sigma)}\mu(J,[1,n]) = -\mu([1,n],[1,n]) -\sum_{J\text{ coatom}}\mu(J,[1,n]) - \mu(K,[1,n]) \\
=& -1 +2 -1 = 0,
\end{align*}
concluding the proof of the theorem. 
\end{proof}

It may seem that \cref{thm:Moebius} only allows to compute the M\"obius function on intervals of $P_\bullet(\sigma)$ whose largest element is the maximum of $P_\bullet(\sigma)$. 
The following remark shows that \cref{thm:Moebius} is actually easily extended to all intervals of $P_\bullet(\sigma)$.

\begin{remark}
Let $\sigma$ be a permutation and $J$ be an element of $P_\bullet(\sigma)$. Define $j=|J|$.
Let $\mathcal{J}$ be the subposet of $P_\bullet(\sigma)$ consisting of the elements in the interval $[\emptyset,J]$ of $P_\bullet (\sigma)$. 
There exists a permutation $\tau$ (of size $j$) such that $P_\bullet(\tau)$ is isomorphic to $\mathcal{J}$.
\end{remark}

\begin{proof}
Let $\hat{\tau}$ be the subsequence of $\sigma$ composed by the elements of $J$. Note that the values occurring in $\hat{\tau}$ form an interval of integers ($J$ being an interval of $\sigma$). 
We then define $\tau$ as the permutation obtained by rescaling $\hat{\tau}$ to the set $\{1,\dots,j\}$. Since the relative order among the elements remains unchanged, the intervals of $\tau$ correspond to the subsets of $J$ that are intervals of $\sigma$. Therefore the poset $P_\bullet(\tau)$ is isomorphic to the poset $\mathcal{J}$.
\end{proof}

As a consequence, for any $I,J\in P_\bullet (\sigma)$, 
we can compute $\mu_{P_\bullet(\sigma)}(I,J)$ using the M\"obius function on $P_\bullet(\tau)$. 
More precisely, letting $I'$ be the interval obtained rescaling $I$ by the same value that we used to rescale $J$ into $[1,j]$, 
we have 
$\mu_{P_\bullet(\sigma)}(I,J) = \mu_{P_\bullet(\tau)}(I',[1,j])$.

For example, let $\sigma=456793128$, whose interval poset $P_\bullet(\sigma)$ is represented in \cref{fig:posets} (adding a minimal element $\emptyset$). If we want to calculate $\mu_{P_\bullet(\sigma)}(\{5\},[4,7])$, we consider $\tau=1234$ (corresponding to $\tau' = 4567$ rescaled by $3$) and compute $\mu_{P_\bullet(\sigma)}(\{5\},[4,7])=\mu_{P_\bullet(\tau)}(\{2\},[1,4])=0$.

\subsection*{Acknowledgements}
We are grateful to Luca Ferrari for several suggestions which improved this paper, and to Bridget Tenner for her feedback on an earlier draft of our paper. 
We also thank both referees, whose comments helped improve the clarity of our work. 

\section*{Declarations}

\subsection*{Author contributions}
All authors contributed equally to the writing of the manuscript.

\subsection*{Conflict of interest}
The authors have no competing interests to declare that are relevant to the content of this article.

\subsection*{Data availability}
Data sharing not applicable to this article as no datasets were generated or analysed during the current study.

\subsection*{Funding Statement}
LC is a member of the INdAM research group GNCS; partially supported by the 2020-2021 INdAM-GNCS project ``Combinatoria delle permutazioni, delle parole e dei grafi: algoritmi e applicazioni''.

\end{document}